\documentclass[12pt,a4paper]{article}

\usepackage{amssymb, latexsym, amsthm}
\usepackage{graphicx} % avoid epsfig or earlier such packages
\usepackage{url}      % for URLs and DOIs
\newcommand{\doi}[1]{\url{http://dx.doi.org/#1}}
\usepackage{amsmath}  % many want amsmath extensions

\IfFileExists{ajr.sty}{\usepackage{ajr}}{}

\newcommand{\p}{\partial}
\newcommand{\D}{\Delta}

\renewcommand{\phi}{\varphi}

\newcommand{\R}{{\mathbb R}}

\newcommand{\spde}{\textsc{spde}}
\newcommand{\sde}{\textsc{sde}}
\newcommand{\pde}{\textsc{pde}}
\newcommand{\rat}[2]{{\textstyle\frac{#1}{#2}}}
\newcommand{\Ord}[1]{\ensuremath{\mathcal O\big(#1\big)}}
\newcommand{\ord}[1]{\ensuremath{o\big(#1\big)}}
\newcommand{\Z}[2][]{e^{-#2 t#1}{\star}}

\newtheorem{theorem}{Theorem}
\newtheorem{lemma}{Lemma}

\title{Averaging approximation to singularly perturbed nonlinear stochastic wave equations}

\author{
Yan Lv\thanks{School of Science, Nanjing University of Science \&
Technology, Nanjing, 210094, \textsc{China}.
\protect\url{mailto:lvyan1998@yahoo.com.cn} }
\and
A.~J. Roberts\thanks{School of Mathematics, University of Adelaide, South Australia, \textsc{Australia}. 
\protect\url{mailto:anthony.roberts@adelaide.edu.au}}
}

\date{\today}

\begin{document}

% Use default \verb|\maketitle|.
\maketitle

% Use the \verb|abstract| environment.
\begin{abstract}
An averaging method is applied to derive effective approximation to
the following singularly perturbed nonlinear stochastic damped wave
equation
\begin{equation*}
 \nu u_{tt}+u_t=\D u+f(u)+\nu^\alpha\dot{W}
\end{equation*}
on an open bounded domain $D\subset\R^n$\,, $1\leq n\leq 3$\,. Here
$\nu>0$ is a small parameter characterising the singular perturbation,
and $\nu^\alpha$\,, $0\leq \alpha\leq 1/2$\,, parametrises the strength of the noise. Some scaling transformations and the martingale representation theorem yield the following effective approximation for small~$\nu$,
\begin{equation*}
u_t=\D u+f(u)+\nu^\alpha\dot{W} 
\end{equation*}
to an error of $\ord{\nu^\alpha}$\,.
\end{abstract}



%\tableofcontents

\section{Introduction}\label{sec:intro}

%The need for taking random  effects into account  in modelling,
%analyzing, simulating and predicting complex phenomena has been
%widely recognized in geophysical and climate dynamics, materials
%science, chemistry,   biology and other areas \cite[e.g.]{ChowB,
%E00,Imkeller}. Stochastic partial differential equations
%(\textsc{spde}s or stochastic \textsc{pde}s) are appropriate
%mathematical models for
% complex systems under random influences \cite{WaymireDuan}.

Wave motion is one of the most commonly observed physical phenomena,
and typically described by hyperbolic partial differential equations. Nonlinear wave equations also have been
studied a great deal in many modern problems such as sonic
booms, bottlenecks in traffic flows, nonlinear optics and quantum
field theory~\cite[e.g.]{ReedSimon, White}. However, for many problems, such
as wave propagation through the atmosphere or the ocean, the
presence of turbulence causes random fluctuations.  More
realistic models must account for such random fluctuations.
Hence we study stochastic wave
equations~\cite[e.g.]{Chow02, ChowB}.

Here we study an effective approximation,  in the sense of
distribution, for the following nonlinear wave stochastic partial differential equation~(\spde).  The \spde\ is a singularly perturbed problem on a bounded open domain $D\subset\R^n$,
$1\leq n\leq3$\,,
\begin{equation}\label{e:SWE}
 \nu u^\nu_{tt}+u^\nu_t=\D u^\nu+f(u^\nu)+\nu^\alpha\dot{W}\,, \quad
 u^\nu(0)=u_0\,,\quad u^\nu_t(0)=u_1\,,
\end{equation}
with zero Dirichlet  boundary on~$D$. Here~$\nu^\alpha$ with
$0<\nu\leq1$ and $0\leq\alpha\leq1/2$ parametrises the strength of noise,
and $\D$~is the Laplace operator in~$\R^n$. The noise~$W(t)$ is an infinite
dimensional Q-Wiener process which is detailed in section~\ref{sec:basic}.  The \spde~(\ref{e:SWE}) also describes the
motion of a small particle with mass~$\nu$ and an infinite number of degrees of freedom~\cite{CF05, CF06}. We are concerned with the effective
approximation of the solution to the \spde~(\ref{e:SWE}) for small $\nu>0$\,.

For $\alpha=1/2$\,, the limit of the random dynamics of \spde~(\ref{e:SWE}) as $\nu\rightarrow 0$ has been studied by Lv and
Wang~\cite{LW08, WL10}. The random attractor and measure attractor
of \spde~(\ref{e:SWE}) are approximated by those of the
deterministic \pde
\begin{equation}\label{1HE}
u_t=\Delta u+f(u)
\end{equation}
as $\nu\rightarrow 0$ in the almost sure sense \cite{LW08} and weak
topology \cite{WL10} respectively.

The important case of $\alpha=0$\,, which is an infinite
dimensional version of the Smolukowski--Kramers approximation, is
studied  by analysing  the structure of solution of linear
stochastic wave equations~\cite{CF05, CF06}. For any $T>0$\,, the
solution~$u(t)$ to the \spde~(\ref{e:SWE}) is approximated in
probability by that of the stochastic system
\begin{equation*}
u_t=\Delta u+f(u)+\dot{W}
\end{equation*}
as $\nu\rightarrow 0$ in space $C(0, T; L^2(D))$.

Here we extend the approximating result to the case when $0\leq
\alpha\leq 1/2$ and derive a higher order approximation in the sense
of distribution. Recently, the stochastic averaging approach was developed
to study the effective approximation to slow-fast \textsc{spde}s
\cite{CerFre09, WangRoberts08} in the following form
\begin{eqnarray*}
u^\nu_t&=&\Delta u^\nu+f(u^\nu, v^\nu)+ \sigma_1\dot{W}_1\,,\\
v^\nu_t&=&\frac{1}{\nu}\left[\Delta
v^\nu+g(u^\nu,v^\nu)\right]+\frac{\sigma_2}{\sqrt{\nu}}\dot{W}_2 \,,
\end{eqnarray*}
where $f$ and~$g$ are nonlinear terms, $\sigma_1$~and~$\sigma_2$ are some
constants, and $W_1$~and~$W_2$ are Wiener processes. Notice that upon introducing
$v^\nu=u^\nu_t$, the \spde~(\ref{e:SWE}) is rewritten as
\begin{eqnarray*}
u^\nu_t&=&v^\nu,\quad  u^\nu(0)=u_0\,,\\
v^\nu_t&=&\frac{1}{\nu}\left[-v^\nu+\Delta
u^\nu+f(u^\nu)\right]+\frac{1}{\nu^{1-\alpha}}\dot{W},\quad
v^\nu(0)=u_1\,,
\end{eqnarray*}
which are also in the form of slow-fast \textsc{spde}s. Then we can
follow the stochastic averaging approach to derive an effective
averaging approximation of~$u^\nu$, the solution of \spde~(\ref{e:SWE}) as
$\nu\rightarrow 0$ for all $0\leq \alpha \leq 1/2$\,. Here the case
$\alpha=1/2$ is the most important case because  all cases of
$\alpha\in [0,1/2]$ can be transformed to the case $\alpha=1/2$\,, see
section~\ref{sec:alpha=0} and section~\ref{sec:alpha=1}.  By an
averaging approach and martingale representation theorem, we prove that
for small $\nu>0$ with $0\leq \alpha\leq 1/2$ the solution of \spde~(\ref{e:SWE}) is approximated in the sense of distribution by~$\bar{u}^\nu$ which solves
\begin{equation}\label{e:1-bar-u}
\bar{u}^\nu_t=\Delta \bar{u}^\nu+f(\bar{u}^\nu)+\nu^\alpha
\dot{\bar{W}}\,, \quad  \bar{u}^\nu(0)=u_0 \,,
\end{equation}
where $\bar{W}(t)$ is a Wiener process distributes same as~$W(t)$.
This result shows that for any small $\nu>0$ with $0\leq \alpha\leq
1/2$ the term~$\nu u^\nu_t(t)$ is a higher order term than the
random force term~$\nu^\alpha W(t)$.

Section~\ref{sec:sqrt-nu} gives the approximation for the important
case that $\alpha=1/2$\,. Previous research~\cite{LW08} gives an
approximation which is a deterministic equation. However, our new
approximation shows that for small $\nu\neq 0$\,, there is a small
fluctuation which  distributes same as~$\sqrt{\nu}W(t)$, see~(\ref{e:1-bar-u}). This gives a more effective approximation.

Section~\ref{ssmcf} explores a parameter regime where a nonlinear coordinate transformation underlies the existence of a stochastic slow manifold for the case $\alpha=0$.   The stochastic slow manifolds of both the \spde~\eqref{e:SWE} and the model~\eqref{e:1-bar-u}  have the same evolution in the parameter regime and so provide evidence of the stronger result of pathwise approximation therein.

\section{Preliminaries}\label{sec:basic}

Let $D\subset \R^n$\,, $1\leq n\leq 3$\,, be a regular domain with
boundary~$\Gamma$. Denote by~$L^2(D)$ the Lebesgue space of square
integrable real valued functions on~$D$, which is a Hilbert space
with inner product
\begin{equation*}
\langle u, v\rangle=\int_Du(x)v(x)\,dx\,, \quad u, v\in L^2(D)\,.
\end{equation*}
Write the norm on~$L^2(D)$ by $\|u\|_0=\langle u, u\rangle^{1/2}$\,.
Define the following abstract operator
\begin{equation*}
Au=-\D u\,,\quad u\in \text{Dom}(A)=\{u\in L^2(D): A u\in L^2(D)\,,\
u|_{\Gamma}=0\}\,.
\end{equation*}
Denote by~$\{\lambda_k\}$ the eigenvalues of~$A$ with
$0<\lambda_1\leq \lambda_2\leq \cdots\leq\lambda_k\leq \cdots$\,,
$\lambda_k\rightarrow\infty$ as $k\rightarrow\infty$\,. For any
$s\in\R$\,,  introduce the space $H^{s}_0(D)=\text{Dom}(A^{s/2})$
endowed with the norm
\begin{equation*}
\|u\|_s=\|A^{s/2}u\|_0\,, \quad u\in H^s_0(D).
\end{equation*}
 Consider the following singularly perturbed
stochastic wave equation with cubic nonlinearity on~$D$:
\begin{eqnarray}\label{SWE1}
   \nu u^\nu_{tt}+u^\nu_t&=&\D u^\nu+\beta u^\nu-(u^\nu)^3+\nu^\alpha\dot{W}(t),\\
   u^\nu(0)&=&u_0\,,\quad  u^\nu_t(0)=u_1\,,\label{SWE2}\\
   u^\nu|_\Gamma&=&0\,,\label{SWE3}
\end{eqnarray}
with $0<\nu<1$ and $0\leq \alpha\leq 1/2$\,. Here $\{W(t)\}_{t\in \R}$ is
an $L^2(D)$-valued two sided Wiener process defined on a complete
probability space ($\Omega,\mathcal{F},\{\mathcal{F}_t\}_{t\geq 0},\mathbb{P}$) with covariance
operator~$Q$ such that
\begin{equation*}
Qe_k=b_k e_k\,, \quad  k=1,2,\ldots\,,
\end{equation*}
where $\{e_k\}$ is a complete orthonormal system in~$L^2(D)$, $b_k$~is a bounded sequence of non-negative real numbers. Then
\begin{equation*}
W(t)=\sum_{k=1}^\infty \sqrt{b_k} e_k w_k(t),
\end{equation*}
where $w_k$ are real mutually independent Brownian
motions~\cite{PZ92}. Further, we assume
\begin{equation}\label{Q}
B_0=\sum^\infty_{k=1} b_k<\infty
\quad\text{and}\quad
B_1=\sum^\infty_{k=1}\lambda_k b_k<\infty\,.
\end{equation}
 Then by
a standard method \cite{Chow07}, for any $(u_0, u_1)\in
  H^{s+1}_0(D)\times H^s(D)$, $s\in\R$\,,
there is a unique solution $u^\nu$ to~(\ref{SWE1})--(\ref{SWE3}),
\begin{eqnarray}\label{sol1}
&&u^\nu\in L^2(\Omega, C(0, T;   H_0^{s+1}(D))), \\
&&u^\nu_t\in L^2(\Omega, C(0, T; H^s(D))).\label{sol2}
\end{eqnarray}
In the following we write $f(u)=\beta u-u^3$ and
$F(u)=\int_0^uf(r)\,dr$\,.

For our purpose we need the following lemma.
%\begin{lemma}{\rm \cite{Lions}}\label{Lions}
%Let $\mathcal{Q}$ be a bounded region in $\R\times \R^n$. For any
%given functions $h_\e$ and $h$ in $L^p(\mathcal{Q})$ $(1<p<\infty)$,
%if
%\begin{equation*}
%|h_\e|_{L^p(\mathcal{Q})}\leq C,\quad  h_\e\rightarrow h\quad  {\rm in}
%\quad \mathcal{Q}\quad  {\rm almost\; everywhere}
%\end{equation*}
%for some positive constant $C$, then $h_\e \rightharpoonup h$ weakly
%in $L^p(\mathcal{Q})$.
%\end{lemma}

\begin{lemma}[Simon~\cite{Simon}]\label{embedding}
Assume $E$, $E_0$ and~$E_1$ be Banach spaces such that $E_1\Subset
E_0$\,, the interpolation space $(E_0, E_1)_{\theta,1}\subset E$ with
$\theta\in (0, 1)$  and $E\subset E_0$ with $\subset$~and~$\Subset$
denoting continuous and compact embedding respectively. Suppose
$p_0,p_1\in [1,\infty]$ and $T>0$\,, such that
\begin{eqnarray*}&&
\mathcal{V} \text{ is a bounded set in } L^{p_1}(0, T;
E_1),\quad\text{and}
\\&&
\p\mathcal{V}:=\{\p v: v\in \mathcal{V}\} \text{ is
a bounded set in } L^{p_0}(0, T; E_0).
\end{eqnarray*}
Here $\p$ denotes the distributional derivative. If
$1-\theta>1/p_\theta$ with
 \begin{equation*}
\frac{1}{p_\theta}=\frac{1-\theta}{p_0}+\frac{\theta}{p_1}\,,
 \end{equation*}
then $\mathcal{V}$ is relatively compact in~$C(0, T; E)$.
\end{lemma}
In the following, for any $T>0$\,, we denote by~$C_T$ a generic
positive constant which is independent of~$\nu$.

\section{The case of  $\alpha=1/2$}\label{sec:sqrt-nu}

We first consider the special case of $\alpha=1/2$ which was recently
studied by a direct approximation method~\cite{LW08, WL10}.
Here we apply an averaging method to give  more effective
approximation to equation~(\ref{SWE1})--(\ref{SWE3}).  We rewrite~(\ref{SWE1})--(\ref{SWE3}) in the form of slow-fast
\textsc{spde}s:
\begin{eqnarray}
du^{\nu}&=&v^{\nu}\,dt\,,\quad u^\nu(0)=u_0\,,\label{3SWE1}\\
dv^{\nu}&=&-\frac{1}{\nu}\left[v^{\nu}-\D
u^{\nu}-f(u^{\nu})\right]dt+\frac{1}{\sqrt{\nu}}\,dW(t)\,, \quad v^\nu(0)=u_1\,.\label{3SWE2}
\end{eqnarray}
Notice that the slow part~$u^\nu$ and fast part~$v^\nu$ are linearly
coupled. For simplicity we consider $(u_0, u_1)\in (H^2(D)\cap
H_0^1(D))\times H^1(D)$. Then~(\ref{SWE1})--(\ref{SWE3}) has a unique
solution in $L^2(\Omega, C(0, T; (H^2(D)\cap H_0^1(D))\times
H^1(D)))$.

\subsection{Tightness of solutions}

Let $(u^{\nu},v^{\nu})$ be a solution to~(\ref{3SWE1})--(\ref{3SWE2}) with
$\nu>0$\,. In order to pass to the limit $\nu\rightarrow 0$ in the averaging
approach, we need some a priori estimates on the solutions.

\begin{theorem}\label{thm:estimate}
Assume $B_1<\infty$\,. For any $T>0$\,, there is a positive constant~$C_T$ such that
\begin{equation}
\mathbb{E}\left[\max_{0\leq t\leq T}\|u^\nu(t)\|_2^2+\max_{0\leq
t\leq T}\| v^\nu(t)\|^2_0 \right]\leq C_T \,,
\end{equation}
and for any integer $m>0$
\begin{equation*}
\mathbb{E}\int_0^T\|u^{\nu}(t)\|_1^{2m}dt\leq C_T\,.
\end{equation*}
\end{theorem}

\begin{proof}
The result on $\|u^\nu(t)\|_2$ is found by a simple energy
estimate~\cite{WL10}.  Now we give the estimate on~$\|v^\nu(t)\|_0$.
By equation~(\ref{3SWE2}),
\begin{eqnarray*}
v^{\nu}(t)&=&e^{-{t}/{\nu}}u_1
+\frac{1}{\nu}\int_0^te^{-({t-s})/{\nu}}\left[\D u^{\nu}(s)+
f(u^{\nu}(s))\right]ds
\\&&{}
+\frac{1}{\sqrt{\nu}}
\int_0^te^{-({t-s})/{\nu}}\,dW(s).
\end{eqnarray*}
Noticing assumption~(\ref{Q}), by the estimate on~$\|u^\nu(t)\|_2$ and maximal inequality of stochastic convolution
\cite[Lemma 7.2]{PZ92},
\begin{equation*}
\mathbb{E}\left[\max_{0\leq t\leq T}\|v^\nu(t)\|_0^2\right]\leq C_T
\end{equation*}
for some positive constant~$C_T$. The last inequality of the theorem
is obtained by the same method~\cite{WL10} and
Poincar\'{e} inequality. This completes the proof.
\end{proof}

Now by the above estimates and Lemma~\ref{embedding}, we have the following theorem.
\begin{theorem}
For any $T>0$, $\{\mathcal{L}(u^\nu)\}_{0<\nu\leq 1}$ the distribution of~$u^\nu$  is
tight in the space $C(0, T; H_0^1(D))$.
\end{theorem}

 By the above tightness result, to determine the limit of~$u^\nu$ we can pass to the limit $\nu\rightarrow 0$ in a weak sense; that is, for any $\phi\in C_0^\infty(D)$,
we consider the limit of  $u^{\nu,\phi}(t)=\langle u^\nu(t), \phi\rangle$ in the space~$C(0, T)$ as $\nu\rightarrow 0$\,.

\subsection{Limit of $u^{\nu, \phi}$ in~$C(0, T)$}
Now we pass to the limit $\nu\rightarrow 0$ in~$\{u^{\nu,\phi}\}$ in the space~$C(0, T)$ for any $T>0$\,. First, by equations~(\ref{3SWE1})--(\ref{3SWE2}),
\begin{eqnarray}
du^{\nu,\phi}&=&v^{\nu,\phi}\,dt\,,\label{phi-SWE1}\\
dv^{\nu,\phi}&=&-\frac{1}{\nu}\left[v^{\nu,\phi}+\langle
\nabla u^{\nu}, \nabla\phi\rangle-\langle f(u^{\nu}), \phi\rangle\right]dt+
\frac{1}{\sqrt{\nu}}\,dW^\phi(t),\label{phi-SWE2}
\end{eqnarray}
with $  u^{\nu,\phi}(0)=\langle u_0, \phi\rangle$ and $
v^{\nu,\phi}(0)=\langle u_1, \phi\rangle$  where
$v^{\nu,\phi}=\langle v^\nu, \phi\rangle$ and $W^\phi(t)=\langle
W(t), \phi\rangle$. In the following we also write~$v^{\nu,\phi}$ as~$v^{\nu,\phi, u(t)}$ which shows the dependence of~$v^{\nu, \phi}$ on the slow part~$u^\nu$.

Second, for any fixed $u\in H^2(D)\cap H_0^1(D)$  we consider the
fast equation
\begin{equation}\label{u-SWE2}
dv^{\nu, u}=-\frac{1}{\nu}\left[ v^{\nu,u}-\Delta u- f(u) \right]dt
+\frac{1}{\sqrt{\nu}}\,dW(t)\,.
\end{equation}
Equation~(\ref{u-SWE2})  has a unique stationary solution~$\bar{v}^{\nu, u}$. Moreover, the stationary solution~$\bar{v}^{\nu,
u}$ is exponentially mixing and the distribution of~$\bar{v}^{\nu, u}$
is the normal distribution
$\mathcal{N}\left( \Delta u+f(u), Q/2  \right)$~\cite{CF05}.

%We also need some result on the differentiability of $v^{\nu, u}$
%with respect to $u$\,. For any $\phi\in H^2\cap H_0^1$\,, let
%$z_\phi^{\nu,  u}=D_uv^{\nu, u}\phi$ be the directional  derivative
%of $v^{\nu, u}$ along the direction $\phi$. Then
%\begin{equation}\label{e:Dv_h}
%\dot{z}_\phi^{\nu, u}=-\frac{1}{\nu}\left[z_\phi^{\nu, u}-\Delta \phi- Df(u) \phi \right]\,,\quad z_\phi^{\nu, u}(0)=0
%\end{equation}
%where $Df(u)$ is the Fr\`echlet derivative of $f$. We have the
%following result.
%\begin{lemma}\label{lem:Dv}
%For any $u\in H^2\cap H_0^1$\,,   $z^{\nu, u}$ is Fr\`echlet
%differentiable with respect to $u$ and the directional derivative
%satisfies,
%\begin{equation*}
%\|z_\phi^{\nu, u}\|_0\leq C(1+\|u\|^2_1)\|\phi\|_2\,, \quad \phi\in
%H_0^1\cap H^2
%\end{equation*}
%for some positive constant $C$\,.
%\end{lemma}

Now for any $u\in H^2(D)\cap H_0^1(D)$ define
\begin{equation*}
H^\nu(u, t)=\nu\left[v^{\nu,u}(t)-v^{\nu,u}(0)\right]+
\int_0^t\left[v^{\nu, u}(s)-\Delta u-f(u)\right]\,ds\,.
\end{equation*}
%and
%\begin{equation*}
%\bar{H}(u, t)= \int_0^{\infty}\mathbb{E}[H(u,t+s)|\mathcal{F}_t]ds=
%\int_t^{\infty}\mathbb{E}[H(u,s)|\mathcal{F}_t]ds
%\end{equation*}
%which is well defined by the exponential mixing of $\bar{v}^{\nu, u}$\,. Here $H(u, t)$ is $H(u, t/\nu)$ with $\nu=1$\,.
Then  $u^{\nu,\phi}$  solves the following equation
\begin{eqnarray}
u^{\nu, \phi}(t)&=&\langle u_0, \phi\rangle-\int_0^t[\langle \nabla
u^{\nu}(s), \nabla\phi \rangle-\langle f(u^{\nu}(s)),
\phi\rangle]\,ds\nonumber\\&& {} +\langle H^\nu(u^\nu(t),t), \phi
\rangle -\nu\left\langle v^{\nu, u}(t)-v^{\nu,
u}(0),\phi\right\rangle\,.\label{e:phi-u}
\end{eqnarray}
Third, we study the behaviour of $\langle H^\nu(u^\nu(t),t), \phi \rangle$ for small~$\nu$.  %We follow a similar discussion by
%Wang and Roberts for slow-fast stochastic ordinary
%equation~\cite{WangRoberts09}.
Let $H^{\nu,\phi}(u, t)=\langle H^\nu(u, t), \phi\rangle $\,, then
define
\begin{eqnarray}\label{e:M-e}
M_t^{\nu,\varphi}= \frac{1}{\sqrt{\nu}}H^{\nu,\varphi}(u^\nu(t),
t)\,.
 \end{eqnarray}
By the definition of~$H^{\nu,\phi}(u, t)$ and equation~(\ref{u-SWE2}), $M_t^{\nu,\varphi}$ is a martingale with respect to
$\{\mathcal{F}_t: t\geq 0\}$, and the quadratic covariance is $\langle M^{\nu,\phi} \rangle_t=t\langle Q\phi, \phi\rangle$.

Now define $R^{\nu,\phi}(t)=-\left\langle v^{\nu, u}(t)-v^{\nu,
u}(0),\phi\right\rangle$,
then rewrite~(\ref{e:phi-u}) as
\begin{equation}\label{e:phi-SWE}
u^{\nu, \phi}(t)=\langle u_0,
\phi\rangle-\int_0^t\left[\langle\nabla u^{\nu}(s),
\nabla\phi\rangle-\langle f(u^{\nu}(s)),
\phi\rangle\right]ds+\sqrt{\nu}M_t^{\nu,\phi}+\nu R^{\nu,\phi}(t)\,.
\end{equation}

Invoking Theorem~\ref{thm:estimate},
\begin{equation}\label{e:R}
\lim_{\nu\rightarrow 0}\mathbb{E}\left[\max_{0\leq t\leq
T}\sqrt{\nu}\left| R^{\nu, \phi}(t)\right|\right]=0\,.
\end{equation}
Then define the process
\begin{equation}\label{e:cal-M-nu}
\mathcal{M}_t^{\nu,\phi}=\frac{1}{\sqrt{\nu}}\left\{
u^{\nu,\phi}(t)-\langle u_0, \phi\rangle +\int_0^t\big
[\langle\nabla u^{\nu}(s), \nabla\phi\rangle-\langle f(u^{\nu}(s)),
\phi\rangle\big] ds \right\}.
\end{equation}
By the definition of~$H^{\nu,\phi}(u, t)$ and~(\ref{e:R}) we have
the tightness of~$\mathcal{M}_t^{\nu,\phi}$ in space~$C(0, T)$ for
any $T>0$\,. Let $P$~be a limit point of the family of probability
measures  $\{\mathcal{L}(\mathcal{M}_t^{\nu,\phi})\}_{0<\nu\leq
 1}$ and denote by~$\mathcal{M}_t^\phi$, a $C(0, T)$-valued random
 variable with distribution~$P$. Let $\Psi$~be a continuous bounded function
 on~$C(0, T)$. Set $\Psi^\nu(s)=\Psi( u^{\nu, \phi}(s) )$, then   noticing~(\ref{e:R}),
\begin{equation*}
\mathbb{E}\left[(\mathcal{M}^{\nu,\phi}_t-\mathcal{M}^{\nu,\phi}_s) \Psi^\nu(s)\right]=
\mathbb{E}\left[\sqrt{\nu} (R^{\nu, \phi}(t)-R^{\nu,\phi}(s)) \Psi^\nu(s)\right]\rightarrow 0\,, \quad  \nu\rightarrow 0 \,,
\end{equation*}
which yields that the process $\{\mathcal{M}_t^\phi\}_{0\leq t\leq
T}$ is a $P$-martingale with respect to the Borel $\sigma$-filter of~$C(0, T)$.

We consider the quadratic covariation of the martingale~$\mathcal{M}_t^\phi$.  By the definition of~$\mathcal{M}^{\nu,\phi}_t$,
 passing to the limit $\nu\rightarrow 0$ in~(\ref{e:cal-M-nu}), we derive~$\mathcal{M}_t^\phi$ is a square integrable martingale with the
associated quadratic  covariation process is~$\langle
Q\phi, \phi\rangle t$. Then by the representation theorem for
martingales~\cite{IW81},  without changing the distributions  of~$\mathcal{M}^{\nu,\phi}_t$ and~$\mathcal{M}_t^\phi$, one extends
the original  probability space~$(\Omega, \mathcal{F}, \mathbb{P})$
and chooses a new Wiener process~$\hat{W}^\phi(t)$  such that $\mathcal{M}_t^\phi=\sqrt{Q} \hat{W}^\phi(t)$, which is unique in the sense of distribution.

By the definition of~$\mathcal{M}_t^{\nu, \phi}$, $\hat{W}^\phi$~can
be chosen as~$\langle \hat{W}, \phi\rangle$ where $\hat{W}$~is a
cylindrical Wiener process. Then from~(\ref{e:cal-M-nu})  we have in the sense of distribution
\begin{eqnarray*}
&&\langle u^\nu(t), \phi\rangle\\&=&\langle u_0, \phi\rangle
-\int_0^t[\langle \nabla u^\nu(s), \nabla\phi\rangle-\langle
f(u^\nu(s)), \phi \rangle] ds
+\sqrt{\nu}\mathcal{M}_t^\phi+\ord{\sqrt{\nu}}\\
&=&\langle u_0, \phi\rangle -\int_0^t[\langle \nabla u^\nu(s),
\nabla\phi\rangle-\langle f(u^\nu(s)), \phi \rangle] ds
+\sqrt{\nu}\sqrt{Q}\langle\hat{W},
\phi\rangle+\ord{\sqrt{\nu}}
\end{eqnarray*}
for any $\phi\in C_0^\infty(D)$. Then by discarding the higher order
term and the tightness of~$u^\nu$,  we have the following
approximating equation
\begin{equation}\label{e:3-bar-u}
d\bar{u}^\nu=[\Delta \bar{u}^\nu+f(\bar{u}^\nu)]dt+\sqrt{\nu} \,
d\bar{W}^Q \,,
\end{equation}
where $\bar{W}^Q$ is some an~$L^2(D)$ valued Q-Wiener process.

\begin{theorem}
Assume $B_1<\infty$ and $\alpha=1/2$\,. For small $\nu>0$\,, there
is a new probability space~$(\bar{\Omega}, \bar{\mathcal{F}},
\bar{\mathbb{P}})$, an extension of the original probability
space~$(\Omega, \mathcal{F}, \mathbb{P})$, such that for any $T>0$\,,
the solution~$u^\nu$ to~(\ref{3SWE1})--(\ref{3SWE2})  is approximated
by~$\bar{u}^\nu$ which solves~(\ref{e:3-bar-u}), to an error of~$\ord{\sqrt{\nu}}$, in the space~$C(0, T; H_0^1(D))$ for
almost all $\omega\in\bar{\Omega}$\,.
\end{theorem}

The above \spde~(\ref{e:3-bar-u}) is more effective than the
limit \pde~(\ref{1HE}) \cite{LW08} as it incorporates
fluctuations for small $\nu>0$\,. This result also implies that the
singular term~$\nu u^\nu_{t}(t)$ is a higher order term than~$\sqrt{\nu}W(t)$ for small $\nu>0$ in the sense of distribution at
least.  The following two sections show that~$\nu
u^\nu_t(t)$ is always a higher order term than~$\nu^\alpha W(t)$ for
any $0\leq \alpha\leq 1/2$\,.

%%%%%%%%%%%%%%%%%%%%%%
%%%%%%%%%%%%%%%%%%%%%%
%%%%%%%%%%%%%%%%%%%%%%

\section{The case of  $\alpha=0$}\label{sec:alpha=0}
Next we consider the case of $\alpha=0$\,; that is, consider the
following \spde
\begin{eqnarray}\label{4-SWE1}
   \nu u^\nu_{tt}+u^\nu_t&=&\D u^\nu+\beta u^\nu-(u^\nu)^3+\dot{W}(t),\\
   u^\nu(0)&=&u_0\,,\quad  u^\nu_t(0)=u_1\,,\label{4-SWE2}\\
   u^\nu|_\Gamma&=&0\label{4-SWE3}\,.
\end{eqnarray}
First we have the following a priori estimates on~$u^\nu$ in the space~$C(0, T; H_0^1(D))$.
\begin{theorem}[Cerrai \& Freidlin~{\cite{CF06}}]\label{CF06}
Assume $B_1<\infty$\,.  For any $T>0$\,, there is a positive constant~$C_T$ such that
\begin{equation*}
\mathbb{E}\left[\max_{0\leq t\leq T}\|u^\nu(t)\|_1^2\right]\leq C_T\,.
\end{equation*}
\end{theorem}

We follow the approach for the case of $\alpha=1/2$\,. For this we
introduce the  scalings $\tilde{u}^\nu=\sqrt{\nu}u^\nu$ and $\tilde{v}^\nu=\sqrt{\nu}u^\nu_t$\,.
Then
\begin{eqnarray*}
d \tilde{u}^\nu&=&\tilde{v}^\nu dt\,,\quad  \tilde{u}^\nu(0)=\sqrt{\nu}u_0\,,\\
d\tilde{v}^\nu&=&-\frac{1}{\nu}\left[\tilde{v}^\nu- \Delta
\tilde{u}^\nu-\sqrt{\nu}f\left(\frac{\tilde{u}^\nu}{\sqrt{\nu}}
\right)\right] dt+ \frac{1}{\sqrt{\nu}} \, dW(t),
\quad \tilde{v}^\nu(0)=\sqrt{\nu}u_1\,.
\end{eqnarray*}
By standard energy estimates~\cite{WL10}, by a similar discussion to that in
Section~\ref{sec:sqrt-nu}, and by Theorem~\ref{CF06}, we have the following theorem.
\begin{theorem}
Assume $B_1<\infty$\,. For any $T>0$\,, there is a positive constant~$C_T$
 such that
\begin{equation*}
\mathbb{E}\left[ \max_{0\leq t\leq
T}\|\tilde{u}^\nu(t)\|_2^2+\max_{0\leq t\leq T}\|
\tilde{v}^\nu(t)\|^2_0\right]\leq C_T \,,
\end{equation*}
and for any integer $m>0$
\begin{equation*}
\mathbb{E}\int_0^T\|\tilde{u}^\nu(t)\|_1^{2m}dt\leq C_T\,.
\end{equation*}
Moreover,
 the distribution of~$\tilde{u}^\nu$ is tight in space~$C(0, T; H_0^1(D))$.
\end{theorem}

%%\begin{remark}
%%In fact by the scaling transformation $\tilde{u}^\nu=\sqrt{\nu}
%%u^\nu$ and Theorem 4.1 we also have the above result.
%%\end{remark}

%We follow the approach for the case of $\alpha=1/2$. For this we
%introduce the following  scalings
%\begin{equation*}
%\tilde{u}^\nu=\sqrt{\nu}u^\nu,\quad
%\tilde{v}^\nu=\sqrt{\nu}u^\nu_t.
%\end{equation*}
%Then we have
%\begin{eqnarray*}
%d \tilde{u}^\nu&=&\tilde{v}^\nu dt,\quad  \tilde{u}^\nu(0)=\sqrt{\nu}u_0\\
%d\tilde{v}^\nu&=&-\frac{1}{\nu}\left[\tilde{v}^\nu- \Delta
%\tilde{u}^\nu-\sqrt{\nu}f\left(\frac{\tilde{u}^\nu}{\sqrt{\nu}}
%\right)\right] dt+ \frac{1}{\sqrt{\nu}} dW(t),
%\quad \tilde{v}^\nu(0)=\sqrt{\nu}u_1.
%\end{eqnarray*}

We consider the asymptotic approximation of~$\tilde{u}^\nu$ for
small $\nu>0$\,. For any $\phi\in  C_0^\infty(D)$, let
$\tilde{u}^{\nu, \phi}=\langle \tilde{u}^\nu, \phi\rangle $,
$\tilde{v}^{\nu, \phi}=\langle \tilde{v}^\nu, \phi\rangle$  and
$W^\phi(t)=\langle W(t), \phi\rangle$. Then
\begin{eqnarray*}
d\tilde{u}^{\nu,\phi}&=&\tilde{v}^{\nu,\phi} dt\,,\\
d\tilde{v}^{\nu,\phi}&=&-\frac{1}{\nu}\left[\tilde{v}^{\nu, \phi}+
\langle \nabla \tilde{u}^\nu, \nabla\phi \rangle-\sqrt{\nu}\langle
f(\tilde{u}^\nu/\sqrt{\nu}), \phi\rangle \right]
dt+\frac{1}{\sqrt{\nu}}\,dW^\phi(t),
\end{eqnarray*}
with $\tilde{u}^{\nu,\phi}(0)=\langle \tilde{u}^\nu(0), \phi\rangle $ and
$\tilde{v}^{\nu,\phi}(0)=\langle \tilde{v}^\nu(0), \phi\rangle $.

We also  consider the following fast \spde\ for fixed~$\nu$
and $\tilde{u}\in H^2(D)\cap H_0^1(D)$:
\begin{equation}\label{e:v-tilde}
d\tilde{v}^{\nu, \tilde{u}}=-\frac{1}{\nu}\left[
\tilde{v}^{\nu,\tilde{u}}- \Delta \tilde{u}-
\sqrt{\nu}f\left(\tilde{u}/\sqrt{\nu}\right) \right]dt
+\frac{1}{\sqrt{\nu}}\,dW(t)\,.
\end{equation}
For fixed $\nu\in (0,1]$ and $\tilde{u}\in H^2(D)\cap H_0^1(D)$,
\spde~(\ref{e:v-tilde}) has a unique stationary solution with the normal
distribution $\mathcal{N}\left( \Delta \tilde{u}+\sqrt{\nu}f(\tilde{u}/\sqrt{\nu}), \; Q/2 \right)$~\cite{CF05}.
%We also need an estimate on $\tilde{z}^{\nu,\tilde{u}}_\phi=D_{\tilde{u}}
%\tilde{v}^{\nu, \tilde{u}}\phi$, the directional derivative of
%$\tilde{v}^{\nu,\tilde{u}}$ along direction $\phi$. In fact
%\begin{equation*}
%\dot{\tilde{z}}^{\nu, \tilde{u}}_\phi=-\frac{1}{\nu}\left[ \tilde{z}^{\nu, \tilde{u}}_\phi
%-\Delta\phi-\sqrt{\nu}f(\tilde{u}/\sqrt{\nu})  \right].
%\end{equation*}
%Then we have the following same result as Lemma~\ref{lem:Dv}
%\begin{lemma}\label{lem:tilde-Dv}
%\begin{equation*}
%\|\tilde{z}^{\nu, \tilde{u}}_\phi\|\leq C\left (1+\frac{1}{\nu}\|\tilde{u}\|^2_1\right)\|\phi\|_2
%\end{equation*}
%for some positive constant $C$.
%\end{lemma}
Now for any $\tilde{u}\in H^2(D)\cap H_0^1(D)$ define
\begin{equation*}
\widetilde{H}^\nu(\tilde{u},t)=\nu\left[\tilde{v}^{\nu,\tilde{u}}(t)-\tilde{v}^{\nu,\tilde{u}}(0)\right]+
\int_0^t\left[\tilde{v}^{\nu, \tilde{u}}(s)-\Delta
\tilde{u}-\sqrt{\nu}f(\tilde{u}/\sqrt{\nu})\right]\,ds\,.
\end{equation*}
%and
%\begin{equation*}
%\widetilde{\bar{H}}(\tilde{u},t)=\int_0^\infty\mathbb{E}[\widetilde{H}(\tilde{u},t+s)|\mathcal{F}_t]ds
%=\int_t^\infty\mathbb{E}[\widetilde{H}(\tilde{u}, s)|\mathcal{F}_t] ds.
%\end{equation*}
Thus  we can follow the same discussion in last section for the case of $\alpha=1/2$\,. %Here one should notice that
%$\widetilde{H}(\tilde{u}^\nu, t)$ is not Lipschitz either. However
%by the tightness of $\tilde{u}^\nu$ and the fact that
%$\tilde{u}^\nu=\sqrt{\nu}u^\nu$, by the estimates for $u^\nu$ we
%still  have the same result.
We write
\begin{eqnarray}
\tilde{u}^{\nu, \phi}(t)&=&\sqrt{\nu}\langle u_0,
\phi\rangle-\int_0^t\langle\nabla\tilde{u}^\nu(s), \nabla\phi
\rangle ds+\sqrt{\nu}\int_0^t\langle
f\left(\tilde{u}^\nu(s)/\sqrt{\nu}\right), \phi\rangle
ds\nonumber\\&&{} +\sqrt{\nu}\tilde{\mathcal{M}}^{\nu, \phi}_t \,,
\label{e:tilde-u}
\end{eqnarray}
where $\sqrt{\nu}\tilde{\mathcal{M}}^{\nu, \phi}_t$ is the remainder
term. By a similar discussion to that of the last section,
$\tilde{\mathcal{M}}^{\nu, \phi}_t$ is tight in space~$C(0, T)$ for
any $T>0$\,. Let $\tilde{P}$~be a limit point of the family of
probability measures $\mathcal{L}\{\tilde{\mathcal{M}}^{\nu,
\phi}_t\}_{0<\nu\leq1}$ in space~$C(0, T)$. Let
$\tilde{\mathcal{M}}^\phi_t$ be a $C(0, T)$-valued random variable
with distribution~$\tilde{P}$. Then we have the following lemma.
\begin{lemma}
For any $\phi\in C_0^\infty(D)$, the process~$\tilde{\mathcal{M}}_t^\phi$ defined on the probability space~$(C(0,
T), \mathcal{B}(C(0, T)), \tilde{P})$ is a square integrable
martingale with the associated quadratic covariation process~$\langle Q\phi, \phi\rangle t$\,.
\end{lemma}

By the representation theorem for martingales~\cite{IW81},  without
changing the distributions  of~$\tilde{\mathcal{M}}^{\nu,\phi}_t$
and~$\tilde{\mathcal{M}}_t^\phi$ one can extend the original
probability space~$(\Omega, \mathcal{F}, \mathbb{P})$ and choose a
new cylindrical Wiener process~$\tilde{W}(t)$  such that $\tilde{\mathcal{M}}_t^\phi=\sqrt{Q} \langle \tilde{W}, \phi\rangle$, which is unique in the sense of distribution.

Then in the sense of distribution by~(\ref{e:tilde-u}) we write out
\begin{eqnarray}
\langle \tilde{u}^\nu(t), \phi\rangle&=&\sqrt{\nu}\langle u_0,
\phi\rangle-\int_0^t\langle \nabla \tilde{u}^\nu(s),
\nabla\phi\rangle ds
+\sqrt{\nu}\int_0^t\langle
f\left(\tilde{u}^\nu(s)/\sqrt{\nu} \right), \phi \rangle ds
\nonumber\\&&{}
+\sqrt{\nu}\tilde{\mathcal{M}}_t^\phi+\ord{\sqrt{\nu}}\nonumber\\
&=&\sqrt{\nu}\langle u_0, \phi\rangle-\int_0^t\langle \nabla
\tilde{u}^\nu(s), \nabla\phi\rangle ds
+\sqrt{\nu}\int_0^t\langle
f\left(\tilde{u}^\nu(s)/\sqrt{\nu} \right), \phi \rangle
ds
\nonumber\\&&{}
+\sqrt{\nu}\sqrt{Q}\langle \tilde{W},
\phi\rangle+\ord{\sqrt{\nu}}
\end{eqnarray}
for any $\phi\in C_0^\infty(D)$. Then we have, noticing that
$\tilde{u}^\nu=\sqrt{\nu}u^\nu$, the following approximating
\spde\ for small $\nu>0$\,:
\begin{equation}\label{e:4-bar-u}
d\bar{u}^\nu=[\Delta\bar{u}^\nu+f(\bar{u}^\nu)]dt+d\bar{W}^Q, \quad
\bar{u}^\nu(0)=u_0\,,
\end{equation}
where $\bar{W}^Q$ is some $L^2(D)$~valued Q-Wiener process. Then
we infer the following result.
\begin{theorem}
Assume $B_1<\infty$ and $\alpha=0$\,. Then for small $\nu>0$\,, there is
a new probability space~$(\bar{\Omega}, \bar{\mathcal{F}},
\bar{\mathbb{P}})$ which is an extension of the original probability
space~$(\Omega, \mathcal{F}, \mathbb{P})$ such that for any $T>0$\,,
the solution~$u^\nu$ to~(\ref{4-SWE1})--(\ref{4-SWE3}) is
approximated by~$\bar{u}^\nu$ which solves~(\ref{e:4-bar-u}), to an 
error of~$\ord{1}$, in the space~$C(0, T; H_0^1(D))$ for
almost all $\omega\in\bar{\Omega}$\,.
\end{theorem}

\section{The case of $0<\alpha<1/2$}\label{sec:alpha=1}

Now we consider the case of $0< \alpha < 1/2$\,; that is,  consider the
following \spde
\begin{eqnarray}\label{5-SWE1}
   \nu u^\nu_{tt}+u^\nu_t&=&\D u^\nu+\beta u^\nu-(u^\nu)^3+\nu^\alpha\dot{W}(t),\\
   u^\nu(0)&=&u_0\,,\quad u^\nu_t(0)=u_1\,,\label{5-SWE2}\\
   u^\nu|_\Gamma&=&0\label{5-SWE3}\,.
\end{eqnarray}

First, by the same analysis as Theorem~\ref{CF06}, we also have the
following result on the a priori estimates on~$u^\nu$.
\begin{theorem}[Cerrai \& Freidlin~{\cite{CF06}}] \label{cfthm}
 Assume $B_1<\infty$\,. For any $T>0$\,, there is a
positive constant~$C_T$ such that
\begin{equation*}
\mathbb{E}\left[\max_{0\leq t\leq T}\|u^\nu(t)\|_1^2\right]\leq C_T\,.
\end{equation*}
\end{theorem}

We also apply the method in Section~\ref{sec:sqrt-nu}. Make the
following scaling transformation $\tilde{u}^\nu=\nu^{1/2-\alpha} u^\nu$ and $\tilde{v}^\nu=\nu^{1/2-\alpha} v^\nu$\,.
Then
\begin{eqnarray*}
d \tilde{u}^\nu&=&\tilde{v}^\nu dt\,,\\
d\tilde{v}^\nu&=&-\frac{1}{\nu}\left[\tilde{v}^\nu- \Delta
\tilde{u}^\nu-\nu^{1/2-\alpha}f\left(\frac{\tilde{u}^\nu}{\nu^{1/2-\alpha}}
\right)\right] dt+ \frac{1}{\sqrt{\nu}} \, dW(t),\\
\tilde{u}^\nu(0)&=&\nu^{1/2-\alpha}u_0\,,\quad \tilde{v}^\nu(0)=\nu^{1/2-\alpha}u_1\,.
\end{eqnarray*}
By a direct energy estimate or the scaling transformation and
Theorem~\ref{cfthm} we deduce the following theorem.
\begin{theorem}
Assume $B_1<\infty$\,. For any $T>0$\,, there is a positive constant~$C_T$  such that
\begin{equation*}
\mathbb{E}\left[ \max_{0\leq t\leq
T}\|\tilde{u}^\nu(t)\|_2^2+\max_{0\leq t\leq T}\|
\tilde{v}^\nu(t)\|^2_0\right]\leq C_T \,,
\end{equation*}
and for any integer $m>0$
\begin{equation*}
\mathbb{E}\int_0^T\|\tilde{u}^\nu(t)\|_1^{2m}dt\leq C_T\,.
\end{equation*}
Moreover, the distribution of~$\tilde{u}^\nu$ is tight in space~$C(0,
T; H_0^1(D))$.
\end{theorem}

Then we can follow the same discussion of Section~\ref{sec:alpha=0}
and have the following result.
\begin{theorem}
Assume $B_1<\infty$ and $0<\alpha<1/2$\,. For small $\nu>0$\,, there is
a new probability space~$(\bar{\Omega}, \bar{\mathcal{F}},
\bar{\mathbb{P}})$ which is an extension of the original probability
space~$(\Omega, \mathcal{F}, \mathbb{P})$ such that for any $T>0$\,,
the solution $u^\nu$ to~(\ref{5-SWE1})--(\ref{5-SWE3}) is
approximated by~$\bar{u}^\nu$ which solves
\begin{equation}\label{e:5-bar-u}
d\bar{u}^\nu=[\Delta\bar{u}^\nu+f(\bar{u}^\nu)]dt+\nu^\alpha
d\bar{W}^Q, \quad  \bar{u}^\nu(0)=u_0\,,
\end{equation} 
to an error of~$\ord{\nu^\alpha}$, in the space~$C(0, T; H_0^1(D))$ for
almost all $\omega\in\bar{\Omega}$.
\end{theorem}

\section{A stochastic slow manifold compares the SPDEs for the case of $\alpha=0$}

\label{ssmcf}

This section shows the long time effectiveness of the averaged model by comparing it to the original via their stochastic slow manifolds.

 We compare the \spde~\eqref{4-SWE1} and its model \spde~\eqref{e:4-bar-u} in a parameter regime where both have an accessible stochastic slow manifold.
Consider the \spde~\eqref{4-SWE1} restricted to one spatial dimension as
\begin{equation}
\nu u_{tt}+u_t=u_{xx}+f(u)+\sigma\dot W
\quad\text{where}\quad
f=(1+\beta')u-u^3.
\label{eq:sde}
\end{equation}
Consider this \spde\ on the non-dimensional domain $D=(0,\pi)$ with boundary conditions $u=0$ on $x=0,\pi$\,.  The parameter~$\sigma$ here explicitly measures the overall size of the Q-Wiener process~$W(t)$ which by~\eqref{Q} is finite. The small parameter~$\beta'$ measures the distance from the stochastic bifurcation that occurs near $\beta'=0$\,.  In this domain there will be a stochastic slow manifold of the \spde~\eqref{eq:sde} that matches the slow dynamics in the approximating \spde~\eqref{e:4-bar-u}.  This section compares the stochastic slow manifolds.

The \spde~\eqref{eq:sde} has a technically challenging spectrum. However, the construction of its stochastic slow manifold is easiest by embedding the \spde~\eqref{eq:sde} as the $\gamma=1$ case of the following slow-fast system of \spde{}s
\begin{align}
u_t={}&u_{xx}+u+v\,,
\label{eq:u}
\\
\nu v_t={}& -v -\gamma \nu\left( \partial _{xx}+1\right)u_t
+\beta' u-u^3
+\sigma\dot W\,.
\label{eq:v}
\end{align}
The parameter~$\gamma$ controls the homotopy: from a tractable base when $\gamma=0$ as then all linear modes in the very fast $v$~equation~\eqref{eq:v} decay at the same rate~$1/\nu$ (and the slow $u$~modes of~$\sin kx$ have decay rates~$1-k^2$); to the original \spde~\eqref{eq:sde} when $\gamma=1$ (upon eliminating~$v$).

\paragraph{A stochastic slow manifold appears}
On the non-dimensional interval~$(0,\pi)$, with Dirichlet boundary conditions on~$u$, the eigenmodes must be proportional to~$\sin kx$ for integer wavenumber~$k$.  Neglecting noise temporarily, $\sigma=0$ in this sentence, for all~$\nu<1$ and all homotopy parameter~$0\leq\gamma\leq1$ there is one zero eigenvalue and all the rest of the eigenvalues have negative real part; the slow subspace corresponding to the neutral mode is spanned by $(u,v)\propto (\sin x,0)$ (local in $(u,v,\sigma)$, but global in $\nu$~and~$\gamma$).  By stochastic center manifold theory~\cite{Arnold03, Boxler89}, and supported by stochastic normal form transformations~\cite{Arnold98, Roberts06k, Roberts07d}, when the noise spectrum truncates and the nonlinearity is small enough, the dynamics of the \spde{}s~\eqref{eq:u}--\eqref{eq:v} are essentially finite dimensional and a stochastic slow manifold exists which is exponentially quickly attractive to all nearby trajectories.

\paragraph{Computer algebra constructs the stochastic slow manifold}
We seek the stochastic slow manifold as a systematic perturbation of the slow subspace $u=a\sin x$\,.  The intricate algebra necessary to handle the multitude of nonlinear noise interactions is best left to a computer~\cite[e.g.]{Roberts05c, Roberts07d}.  However, the following expressions may be checked by substituting into the governing \spde{}s~\eqref{eq:u}--\eqref{eq:v} and confirming the order of the residuals is as small as quoted---albeit tedious, this check is considerably easier than the derivation.  The evolution on the stochastic slow manifold may be written
\begin{align}
\dot a={}& \beta' a-\rat34a^3
+\left[ 1-2\nu\beta'
+\rat92\nu a^2
-\rat9{1024}a^4
\right]b_1\dot w_1
\nonumber\\&{}
+\left[(\rat3{32}+\rat3{128}\beta')a^2
-\rat{21}{1024}a^4\right]b_3\dot w_3
+\rat5{1024}a^4b_5\dot w_5
+\ord{\nu^2+{\beta'}^2+a^4,\sigma}
\label{eq:ssme}
\end{align}
The  stochastic slow manifold itself involves Ornstein--Uhlenbeck processes written as convolutions over the past history of the noise processes: define $\Z\mu\dot w=\int_{-\infty}^t\exp[-\mu(t-s)]dw_s$ for decay rates $\mu_k=k^2-1$ characteristic of the $k$th~mode.  Then the stochastic slow manifold is
\begin{align}
u={}& a\sin x +\rat1{32}a^3\sin3x
-\rat3{32}a^2\left[b_3\Z8w_3\sin x+b_1\Z8w_1\sin 3x\right]
 \nonumber\\&{}
  +\sum_{k\geq2}b_k
 \left[1+\mu_k\nu+\gamma\nu(\mu_k-\mu_k^2\Z{\mu_k})\right]
 \Z{\mu_k}\dot w_k\sin kx
 \nonumber\\&{}
  -\sum_{k\geq1}b_k\Z[/\nu]{}\dot w_k\sin kx
  +\beta'\sum_{k\geq2}b_k\Z{\mu_k}\Z{\mu_k}\dot w_k\sin kx
 \nonumber\\&{}
 +\rat34\sum_{k\geq2}\left\{
 b_{k+2}\Z{\mu_k}\Z{\mu_{k+2}}\dot w_{k+2}\sin kx
 -2b_k\Z{\mu_k}\Z{\mu_{k}}\dot w_{k}\sin kx
 \right.\nonumber\\&\left.\qquad{}
 +b_{k}\Z{\mu_{k+2}}\Z{\mu_{k}}\dot w_{k}\sin[(k+2)x]
 \right\}
+\Ord{\nu^2+{\beta'}^2+a^4,\sigma^2},
\label{eq:ssmu}
%v\approx{}&
%\frac\sigma\nu\left[
%\Z[/\nu]{}\dot w_1\sin x
%+\Z[/\nu]{}\dot w_2\sin 2x
%+\Z[/\nu]{}\dot w_3\sin 3x \right]
%+\rat12\beta a^2\sin 2x
%\nonumber\\&
%+\frac{\gamma\sigma}\nu\left[
%3\Z[/\nu]{}\Z[/\nu]{}\dot w_2\sin 2x
%+8\Z[/\nu]{}\Z[/\nu]{}\dot w_3\sin 3x\right]
\end{align}
and a correspondingly complicated expression for the field~$v(x,t)$.
Observe that the slow \sde~\eqref{eq:ssme} does not contain any fast time convolutions from the Ornstein--Uhlenbeck processes: it would be incongruous to have such fast processes in a supposedly slow model.  We keep fast time convolutions out of the slow \sde~\eqref{eq:ssme} by introducing carefully crafted terms in the slow mode~$\sin x$ in the parametrization of the stochastic slow manifold~\eqref{eq:ssmu}: here the amplitude of the slow mode~$\sin x$ is approximately
$a-\rat3{32}a^2b_3\Z8w_3-b_1\Z[/\nu]{}\dot w_1$\,.  Other methods which do not adjust the slow mode either average over such adjustments and so are weak models, or invoke fast processes in the slow model.

Note that the homotopy parameter~$\gamma$ affects the stochastic slow manifold shape~\eqref{eq:ssmu}, but only weakly.  To this order the homotopy has no effect on the evolution on the stochastic slow manifold~\eqref{eq:ssme}.

\paragraph{Compare with SPDE \eqref{e:4-bar-u}}
The corresponding stochastic slow manifold of the \spde~\eqref{e:4-bar-u}, in this parameter regime, is straightforward to construct, via the web server~\cite{Roberts07d} for example.  For stochastic slow manifold $\bar u\approx \bar a\sin x$ one finds the corresponding slow \sde\
\begin{align}
\dot {\bar a}={}& \beta' \bar a-\rat34\bar a^3 %+\rat3{128}\bar a^5
+\left[ 1-\rat9{1024}a^4
\right]\bar b_1\dot{\bar w}_1
\nonumber\\&{}
+\left[(\rat3{32}+\rat3{128}\beta')\bar a^2
-\rat{21}{1024}a^4\right]\bar b_3\dot{\bar w}_3
+\rat5{1024}a^4\bar b_5\dot{\bar w}_5
+\ord{{\beta'}^{2}+\bar a^4,\sigma}.
\label{eq:ssmm}
\end{align}
This slow \sde\ is symbolically identical with the \sde~\eqref{eq:ssme}, one just removes the overbars.  We conclude that these stochastic slow manifolds confirm the modeling of the \spde~\eqref{4-SWE1} by its model \spde~\eqref{e:4-bar-u}; at least in the regime of one space dimension with small amplitude~$a$, bifurcation parameter~$\beta'$, and finite truncation to the noise.

%\section{Conjecture}
%Notice that our scaling transformation in section~\ref{sec:alpha=0}
%and section~\ref{sec:alpha=1} is not applicable to the case
%$\alpha>1/2$, but we conjecture that our result hold for all
%$\alpha\geq 0$. This will be done in future work.

\paragraph{Acknowledgements}  The research was supported  by the
NSF of China grant  No.~10901083, Zijin star of Nanjing University of Science and Technology, and the ARC grant DP0988738.

\end{document}